\documentclass[12pt,twoside]{article}
\usepackage{amssymb,amsmath,amsthm,latexsym}
\usepackage{amsfonts}
\usepackage{amsfonts}
\usepackage{graphicx}
\usepackage[pdftex,bookmarks,colorlinks=false]{hyperref}
\usepackage{eucal,enumerate}

\newcommand\eset{\emptyset}

\newtheorem{thm}{Theorem}[section]
\newtheorem{prop}[thm]{Proposition}
\newtheorem{cor}[thm] {Corollary}
\newtheorem{lem} [thm]{Lemma}
\theoremstyle{definition}
\newtheorem{dfn}[thm]{Definition}
\newtheorem{ex}[thm]{Example}

\def\ni{\noindent}

\begin{document}
\label{'ubf'}
\setcounter{page}{1} 
\markboth {\hspace*{-9mm} \centerline{\footnotesize \sc
    Interference in Graphs}
                 }
                { \centerline {\footnotesize \sc
        Acharya, Germina, Kurian, Paul and Zaslavsky} \hspace*{-9mm}
               }

\thispagestyle{empty}
\begin{center}
{
       {\Large \textbf { \sc Interference in Graphs
    }
       }

\bigskip

{\sc B.D.\ Acharya\footnote{ \underline{Current Address}: Wrangler D.C. Pavate Institute of Mathematical Sciences, Karnatak University, Pavate Nagar, Dharwad-580003, Karnataka, India.}} \\
\medskip
{\footnotesize Srinivasa Ramanujan Center for Intensification
of Interaction in Interdisciplinary \\ Discrete
Mathematical Sciences (SRC-IIIDMS), \\University of Mysore, Mysore - 570005, \ India}}\\
{\footnotesize{devadas.acharya@gmail.com}}\\
\centerline{}
{\sc Germina K.A.\footnote{ \underline{Current Address}: Department of Mathematics, Central University of Kerala, Kasaragode, Kerala, India.}} \\
\medskip
{\footnotesize Research Centre \& PG Department of Mathematics, \\ Mary Matha Arts \& Science College, \\ Vemom P.O., Mananthavady - 670645, \ India.}\\
{\footnotesize{srgerminaka@gmail.com }}\\
\centerline{}
{\sc Rency Kurian} \\
\medskip
{\footnotesize Department of Mathematics,\\ Nirmalagiri  College, Nirmalagiri, Kerala,  \ India.}\\  {\footnotesize{rencykuryan@gmail.com}}\\
\centerline{}
{\sc Viji Paul} \\
\medskip
{\footnotesize Department of Mathematics, \\ WMO Arts \& Science  College, Muttil,
Kerala,  \ India.}\\  {\footnotesize{vijipaul10@gmail.com}}
\centerline{}
{\sc Thomas Zaslavsky}\\
 \medskip
{\footnotesize Department of Mathematical Sciences,\\ Binghamton University (SUNY), \\Binghamton, NY 13902-6000, U.S.A.}\\  {\footnotesize{zaslav@math.binghamton.edu}}

\end{center}
\newpage

\begin{center}
\begin{abstract}	

 Given a graph $I=(V, E),$ $\emptyset \ne D \subseteq V,$ and an arbitrary nonempty set $X,$ an injective function  $f: V\to 2^X \setminus \{\emptyset\}$ is an \emph{interference} of $D$ with respect to $I,$ if for every vertex $u\in V\setminus D$ there exists a neighbor $v\in D$ such that $f(u)\cap f(v) \ne \emptyset.$  We initiate a study of interference in graphs.  We study special cases of the difficult problem of finding a smallest possible set $X,$ and we decide when, given a graph $G=(V,E(G))$ (resp., its line graph $L(G)$) the open neighborhood function $N_G: V \to 2^V$ (resp., $N_{L(G)}: E \to 2^E$) or its complementary function is an interference with respect to the complete graph $I=K_n.$
\end{abstract}
\end{center}

\bigskip

\ni\textbf{Mathematics Subject Classification 2010:} Primary 05C78; Secondary 05C76.
\\
\ni\textbf{Key words:} interference in graphs, neighborhood interference, interference index, line graph, distance-pattern distinguishing set.
\vspace{0.30cm}

\section{Interference of subsets of vertices in a graph}

For terminology and notation in graph theory, we refer the reader  to F.\ Harary~\cite{fh}.  The graphs considered in this paper are finite and simple.

We are interested in injective labelings of the vertices of a graph by sets (that is, injective `vertex set-valuations', or `set-labelings'; see \cite{bda1}) of a new kind, such that the sets labeling a certain subset of the vertices \emph{interfere} with those of all other vertices, in the sense captured by the following definition.

\begin{dfn} \label{df.1} Let there be given a graph $I=(V,E(I))$ (the \emph{interference graph}).  An injective function $f: V\to 2^X \setminus \{\emptyset\},$ where $X$ is an arbitrary nonempty set (the \emph{ground set}), is an \emph{interference} of a set $\emptyset\ne D \subseteq V$ with respect to $I,$ if for every vertex $u\in V\setminus D$ there exists a vertex $v\in D \cap N(u)$ such that
$$
f(u)\cap f(v) \ne \emptyset.
$$
Here $N(u)$ denotes, as usual, the (open) neighborhood of $u,$ i.e., $N(u):=\{v \in V: v \text{ is adjacent to } u\}.$

Given a family $\mathcal{P}$ of nonempty subsets of $V,$ we say $f$ is a \emph{$\mathcal{P}$-interference with respect to $I$} if it is an interference for every set $D \in \mathcal{P}.$
\hfill$\blacksquare$
\end{dfn}

We apply the name `interference' because we think of overlapping sets $f(u)$ and $f(v)$ as `interfering' with each other if $u$ and $v$ are adjacent in $I.$

We are interested in three main problems.  First, we want a common interference function of every dominating set of vertices in the interference graph $I;$  we call this a `universal interference with respect to $I.$'  It is easy to show that such a function exists.  Next, we want to know the smallest size of a ground set for which that is possible.  Third, we study two kinds of example; we want to know when the labeling $N$ that assigns to a vertex its neighborhood in a graph $G$---we call this `neighborhood interference'---or the labeling $\bar N$ that assigns the complement of the neighborhood is an interference, if the interference graph is the complete graph, $I=K_n.$  We also ask the same question for the line graph $L(G).$  We treat a number of illustrative examples.  We conclude with a brief look at a more abstract example, where the ground set $X$ is a set of numbers associated with distance in a graph $G,$ again in the special case where $I$ is a complete graph; and with open problems.

\section{Basic properties of interference}\label{basic}

A nonempty set $D$ of vertices in a graph $I$ is called a \emph{dominating set} of $I$ if every vertex in $I$ is either in $D$ or adjacent to a vertex in $D$ \cite{hhs}.  We write $n$ for the order of $I.$

\begin{lem}\label{domset1}
Given an interference graph $I,$ a set $D \subseteq V$ has an interference only if it is a dominating set of $I.$
\end{lem}

\begin{proof}
If $D$ is not a dominating set, there is a vertex $u \in V \setminus D$ which has no neighbor in $D.$  Therefore, no vertex $v \in D$ can satisfy $v \in D \cap N(u).$  It follows that $D$ cannot have an interference.
\end{proof}

Lemma \ref{domset1} demonstrates that we cannot have an interference of a non-dominating set.  The most we can expect of a function $f: V \to 2^V \setminus \{\emptyset\}$ is that it be an interference of every dominating set.

\begin{dfn}\label{ipattern} We say a function $f: V \to 2^V \setminus \{\emptyset\}$ is \emph{universal with respect to $I$} if it is an interference of every dominating subset of vertices in $I.$  If $f$ is universal with respect to the complete graph on vertex set $V,$ that is, if $f(u)\cap f(v) \neq \emptyset$ for any pair $u,v$ of distinct vertices, we simply call it a \emph{complete interference} (for $V$).	 \hfill $\blacksquare$
\end{dfn}

\begin{thm} \label{univ1} Every interference graph $I$ possesses a universal interference.  In particular, a complete interference for $V$ is a universal interference for every interference graph on vertex set $V.$
\end{thm}

\begin{proof} The second statement is obvious.  For the first, we present a complete interference for $V$ such that $|X|=1+\lceil\log_2 n\rceil.$ In any such $X,$ choose an element $x_0 \in X$ and define $f:V\to 2^X \setminus \{\emptyset\}$ by letting $f(v),$ $v \in V,$ be any $n$ subsets of $X$ that contain $x_0.$  Then the sets $f(v),$ $v\in V,$ form an intersecting family and therefore define a complete interference.
\end{proof}

\begin{thm}\label{domset}
Given an interference graph $I,$ a set $D \subseteq V$ has an interference if and only if it is a dominating set of $I.$
\end{thm}

\begin{proof}
If $D$ is a dominating set, a complete interference for $V$ is an interference of $D$ with respect to any interference graph $I.$  For the other direction we have Lemma \ref{domset1}.
\end{proof}

\begin{prop}\label{P:universality}
Given an interference graph $I,$ an injective function $f: V\to 2^X \setminus \{\emptyset\}$ is a universal interference with respect to $I$ if and only if it is an interference of every minimal dominating set in $I.$
\end{prop}

\begin{proof}
Assume the latter. Let $D$ be any dominating set in $I$ and choose a minimal $D' \subseteq D.$  Then for every $u \notin D,$ there is $v \in D' \subseteq D$ such that $f(u) \cap f(v) \neq \emptyset.$
\end{proof}

We note that in the complete graph, a minimal dominating set is a singleton set $\{v\},$ $v \in V.$

\section{Interference index}\label{index}

We proved in Theorem \ref{domset} that a class $\mathcal{P} \subseteq 2^X \setminus \{\emptyset\}$ has an interference if and only if it contains a dominating set in $I.$
It is natural to ask, given $I$ and $\mathcal{P},$ for the size of a smallest ground set that admits an interference of $\mathcal{P}$ with respect to $I.$  We call this size the \emph{interference index} of $\mathcal{P}$ with respect to $I,$ written $i(I,\mathcal{P}).$  (When $\mathcal{P}$ contains a single set $D,$ we write $i(I,D).$)  The \emph{universal interference index} with respect to $I,$ denoted by $i(I),$ is the interference index of $\mathcal{P} = 2^V \setminus \{\eset\}.$

We define $n := |V|.$  A notation that is useful is, for $D \subseteq V,$ to define $f(D) := \bigcup_{u\in D} f(v).$  The condition that $f(u)\cap f(v) \ne \eset$ for some $v\in D$ can be restated as $f(u) \cap f(D) \ne \eset.$

\begin{lem} \label{minindex} For any interference graph $I$ and for any interference $f: V \to 2^X \setminus \{\emptyset\},$ we have $|X| \ge \lceil \log_2(n+1)\rceil.$
\end{lem}

\begin{proof}
The set $X$ has a total of $2^{|X|}-1$ nonempty subsets.  Since $f$ is injective, $n\le 2^{|X|}-1;$ therefore $2^{|X|}\ge n+1,$ hence $|X|\ge \log_2(n+1).$  $|X|$ being an integer, $|X|\ge \lceil \log_2(n+1)\rceil.$
\end{proof}

\begin{thm} \label{univ index} For any interference graph $I$ with $|V|=n$ and for any class $\mathcal{P}$ of nonempty subsets of $V$ that contains a dominating set of $I,$ the interference index satisfies $\lceil \log_2(n+1)\rceil\le i(I,\mathcal{P}) \le \lceil \log_2 (2n)\rceil.$

The universal interference index with respect to $I=K_n$ is $i(K_n)=\lceil\log_2 (2n)\rceil.$
\end{thm}

\begin{proof} Lemma \ref{minindex} implies the lower bound in the first part.

We recall from extremal set theory (see \cite{And}) that an \emph{intersecting family of sets} is a family of sets, no two of which have empty intersection, and that the largest size of an intersecting family of subsets of a $k$-element set is $2^{k-1}.$
For the universal interference index when $I=K_n,$ let $u,v \in V$ and $u \ne v.$  Because $\{v\} \in \mathcal{P},$ we must have $f(u) \cap f(v) \ne \eset.$  That is, $f$ is a universal interference with respect to $I$ if and only if the family $\{f(u) : u \in V\}$ is an intersecting family.  The largest size of an intersecting family, $2^{|X|-1},$ must be at least $n;$ that is, $n \leq 2^{|X|-1},$ or $|X| \ge 1+\log_2 n.$  As $|X|$ is an integer, $|X| \ge 1+\lceil\log_2 n\rceil.$

The construction of Theorem \ref{univ1} defines a complete interference with $|X|=\lceil\log_2 n\rceil+1.$  Thus, $i(K_n)=\lceil\log_2 n\rceil+1.$  The existence of this complete interference implies the upper bound in the first part.
\end{proof}

The bounds of Theorem \ref{univ index} show that the value of an interference index has the form $\lceil \log_2(n+k) \rceil$ where $1 \le k \le n.$  In any one example there may be several possible values of $k,$ but for a family of examples we hope to find a single value of $k$ that gives the index for the entire family.  The next few results illustrate this.

\begin{cor}\label{power of 2}
Suppose $n = 2^m.$  Then, for any interference graph $I$ with $|V|=n$ and for any class $\mathcal{P}$ of nonempty subsets of $V$ that contains a dominating set of $I,$ $i(I,\mathcal{P}) = m+1.$
\end{cor}

\begin{prop}\label{dom set}
Given an interference graph $I,$ let $D \subseteq V$ be nonempty and such that every vertex in $D$ is adjacent in $I$ to every vertex in $V \setminus D.$  Then the interference index of $D$ with respect to $I$ is $i(I,D) = \lceil\log_2(n+1)\rceil.$
\end{prop}

\begin{proof}
The requirement on an interference for $D$ is that $f(w) \cap f(D) \ne \eset$ for all $w \notin D.$  It does not matter how the elements of $f(D)$ are distributed among the vertices in $D.$  Thus, we may ignore $X \setminus f(D)$ and simply assume $f(D) = X;$ then the requirement is that the $s$ sets $f(w),$ $w \notin D,$ are nonempty.  For instance, we may choose $u_0\in D,$ set $f(u_0) = X,$ and let all other sets $f(v)$ be distinct nonempty proper subsets of $X.$  
Thus, we find that $X$ in our construction need only satisfy $|X| \ge \lceil \log_2(n+1)\rceil,$ implying that $i(I,D) \le \lceil\log_2(n+1)\rceil.$  The opposite inequality is from Theorem \ref{univ index}.
\end{proof}

For instance, $i(I,\{u\}) = \lceil\log_2(n+1)\rceil$ for a dominating vertex $u$ in the interference graph.

The proof of Theorem \ref{univ index} shows that interference index leads to extremal set theory.  Here is another example.  Let $b_r(m)$ be the largest number $s$ for which there exist $r+s$ distinct subsets of an $m$-element set $X,$ such that every one of the first $r$ subsets intersects every one of the last $s$ subsets.  (The values of $b_r(m)$ are not known, except for small values of $r.$)

\begin{thm}\label{krs}
For the interference graph $I=K_{r,s},$ of order $n=s+r,$ with $r\le s,$ the universal interference index satisfies $i(K_{r,s}) \le \lceil\log_2(n+r)\rceil.$
Equality holds when $r\le4.$
In general, $i(K_{r,s}) =$ the smallest $m$ such that $s \le b_r(m).$

If the vertex bipartition of $K_{r,s}$ is $V = U \cup W,$ then $i(K_{r,s},U) = i(K_{r,s},W) = \lceil\log_2(n+1)\rceil.$ 
Furthermore, the interference index of any class $\mathcal{P}$ that contains every pair $\{u,w\},$ $u\in U$ and $w\in W,$ is $i(K_{r,s},\mathcal{P}) = i(K_{r,s}).$
\end{thm}

\begin{proof}
Let the two sides of the vertex set be $U$ and $W$ with $|U|=r$ and $|W|=s.$  The requirements for an interference $f$ are that $f(u) \cap f(W) \ne \eset,$ $u\in U,$ and $f(w) \cap f(U) \ne \eset,$ $w \in W.$

We prove the general inequality by describing an interference with $|X|=\lceil\log_2(n+r)\rceil.$  Choose $f(U) \subset 2^X$ to be an order filter; that is, if $Y \in f(U),$ then every $Z \subseteq X$ such that $Y \subset Z$ is also in $f(U).$  In particular, then $X \in f(U).$  Let $\overline{f(U)} := \{ X \setminus Y : Y \in f(U)\}.$  
For $w \in W,$ we must have $f(w) \cap f(u) \ne \emptyset$ for every $u \in U.$  Equivalently, $f(w) \notin \overline{f(U)}.$  Since $f(w) \notin f(U)$ by injectivity of $f,$ we may choose $f(W) \subseteq 2^X \setminus \big(f(U) \cup \overline{f(U)}\big),$ which is a class of size not less than $2^{|X|} - 2r.$  
Therefore, if $s \leq 2^{|X|} - 2r,$ an interference exists with ground set $X.$  This sufficient condition can be rewritten as $2^{|X|} \ge n+r$ or, equivalently, $|X| \ge \lceil\log_2(n+r)\rceil.$  The minimum possible $|X|$ therefore satisfies $i(K_{r,s}) \le \lceil\log_2(n+r)\rceil.$

The equality for $r=1$ follows from Theorem \ref{univ index}.

Now let $r\ge2.$  The only minimal dominating sets in $K_{r,s}$ other than $U,$ $W$ are the pairs $\{u,w\}$ with $u\in U$ and $w\in W.$  If $f$ is an interference for all pairs $\{u,w\},$ it follows that $f(u) \cap f(W) \ne \eset,$ $u\in U,$ and $f(w) \cap f(U) \ne \eset,$ $w \in W;$ hence, $f$ is a universal interference with respect to $I=K_{r,s}.$
We deduce that $i(K_{r,s},\mathcal{P}) = i(K_{r,s})$ for any class $\mathcal{P}$ as described in the theorem.

To prove that $i(K_{r,s}) =$ the smallest $m$ such that $s \le b_r(m),$ consider a universal interference $f$ with respect to $K_{r,s}$ having ground set $X$ of size $m.$  Note that the dominating set $\{u_1,w_1\},$ $u_1\in U$ and $w_1\in W,$ implies that $f$ must satisfy $f(u) \cap f(w_1) \ne \eset$ for all $u \in U\setminus \{u_1\}$ and $f(u_1)\cap f(w) \ne \eset$ for all $w \in W\setminus \{w_1\}.$  Since $r,s\ge2,$ we conclude that all intersections $f(u)\cap f(w)$ are nonempty.  Let $U = \{u_1,\ldots,u_r\}$ and $W = \{w_1,\ldots,w_s\}.$  The sets $f(u_1),\ldots,f(u_r),f(w_1),\ldots,f(w_s)$ are $r+s$ sets as in the definition of $b_r(m).$  Hence, $s \le b_r(m).$  Taking the smallest possible $X,$ that is with $m = i(K_{r,s}),$ we see that $i(K_{r,s})$ must satisfy $s \le b_r(i(K_{r,s})).$
Conversely, if $s \le b_r(m)$ and we have sets $X_1,\ldots,X_r,X_{r+1},\ldots,X_{r+s} \subseteq X$ as in the definition, then defining $f(u_i)=X_i$ and $f(w_j)=X_{r+j}$ gives an interference with ground set $X$ of size $m,$ whence $i(K_{r,s}) \le m$ for any $m$ such that $s \le b_r(m).$  It follows that $i(K_{r,s}) =$ the smallest $m$ such that $s \le b_r(m).$

For $r=2$ let $f(U)=\big\{X,X\setminus\{a\}\big\},$ where $a$ is any one element of $X.$  The sets $f(w)$ should be any sets $Y \subseteq X,$ different from $X$ and $X\setminus\{a\},$ that are not contained in $\{a\}.$  There are $2^m-4$ such sets; therefore, $s \le 2^m-4$ and $n\le 2^m-2.$  We deduce that $|X| \ge \log_2(n+2)$ and that $|X|=\lceil\log_2(n+2)\rceil$ does give a universal interference.  Thus, $i(K_{2,s}) = \lceil\log_2(n+2)\rceil.$

The proof of equality for $r=3,4$ is more complicated and is omitted.

For the values of $i(K_{r,s},U)$ and $i(K_{r,s},W)$ we apply Proposition \ref{dom set}.
\end{proof}

\section{Neighborhood-based interference}\label{nbd}

If we have a graph $G=(V,E)$ on the same vertex set $V$ as the interference graph $I,$ then the fact that $N_G(u)$ is defined for every $u \in V$ makes $N_G$ a function $V \to 2^V.$  In this section we consider the interference character of the \emph{neighborhood function} $N$ defined by $N(u):=N_G(u),$ and the \emph{complemented neighborhood function} $\bar N$ defined by $\bar N(u) := \bar N_G(u) := V \setminus N_G(u).$
We assume throughout this section that \emph{the interference graph $I$ is the complete graph} $K_n;$ thus, a universal interference with respect to $I$ means a complete interference.

We write $\langle X \rangle_G$ for the induced subgraph of $G$ on $X \subseteq V.$
We call $G$ \emph{point-determining} if $N$ is injective (cf.\ \cite{dps}).
The \emph{distance} between $u$ and $v,$ $d(u,v),$ is the length of a shortest path between them in $G;$ if there is no such path $d(u,v)=\infty.$  The distance from $u$ to a nonempty set $D \subseteq V(G)$ is defined as $d(u,D) = \min_{v \in D} d(u,v).$

\begin{lem}\label{L:injectivity}
Each of the functions $N$ and $\bar N$ is injective if and only if the graph $G$ on which they are defined is point-determining.
\end{lem}

\begin{proof}It is clear that $N$ is injective if and only if $G$ is point-determining. For $u \neq v,$ $N(u)=N(v) \Leftrightarrow V\setminus N(u)=V\setminus N(v) \Leftrightarrow \bar N(u)=\bar N(v).$ Therefore, $N$ is injective if and only if $\bar N$ is injective.
\end{proof}

\begin{lem}\label{L:nonisolated}
The empty set is not a value of $N$ if and only if $G$ has no isolated vertices.  It is never a value of $\bar N.$
\end{lem}

\subsection{Neighborhood interference}

In this section we let $D$ be a nonempty subset of $V$ and we characterize the graphs $G$ such that $N$ is an interference of $D.$

We define $N^2(u):=\{w\in V:d(u,w)=2\}.$

\begin{thm}\label{T:nbd}
Let $G$ be a graph and $\eset \neq D \subseteq V.$  The neighborhood function $N$ is an interference of $D$ if and only if $G$ is point-determining and has no isolated vertices and, for every $u \in V \setminus D,$
\begin{enumerate}[\rm(a)]
\item $d(u,D) \leq 2$ and
\item if $N^2(u) \cap D = \eset,$ then there exists $v \in D \cap N(u)$ such that $u,v$ are contained in a triangle; equivalently, not all members of $D \cap N(u)$ are isolated in the induced subgraph $\langle N(u) \rangle_G.$
\end{enumerate}
\end{thm}

\begin{proof}
If $d(u,D) > 2,$ then for every $v \in D,$ $N(u)$ and $N(v)$ do not overlap; thus $N$ is not an interference.  Assume now that $d(u,D) \leq 2$ for every $u \notin D.$

If there is $v \in D$ such that $d(u,v) = 2,$ then $N(u) \cap N(v)$ is nonempty.

If $N^2(u) \cap D = \eset,$ then $D \cap N(u) \neq \eset.$  For $N$ to be an interference there must exist $v \in D$ such that $N(v) \cap N(u)$ is nonempty.  Thus $d(u,v) \leq 2$ and since nothing in $D$ has distance 2 from $u,$ $v \in N(u).$

Now let $v \in D \cap N(u).$  If $v$ is not isolated in $\langle N(u) \rangle_G,$ there is an edge $vw \in E(\langle N(u) \rangle_G)$ and $w \in N(u) \cap N(v),$ so the overlap requirement on $u$ is satisfied.  If every $v$ is isolated, then $N(v) \cap N(u)$ is empty for every $v,$ thus, the overlap requirement on $u$ is not satisfied, and $N$ is not an interference.

A vertex $v \in N(u)$ is not isolated in $\langle N(u) \rangle_G$ if and only if $u$ and $v$ have a common neighbor; equivalently, the edge $uv$ lies in a triangle.
\end{proof}

\begin{cor}\label{C:nbdsingleton}
Let $v \in V(G).$  The neighborhood function $N$ is an interference of $\{v\}$ if and only if $G$ is point-determining, $|V|\ge2,$ $d(u,v) \leq 2$ for every $u \in V,$ and $\langle N(v) \rangle_G$ has no isolated vertices.
\end{cor}

\begin{proof}
The overlap condition from Theorem \ref{T:nbd} is that $d(u,v) \leq 2$ for every $u \in V \setminus v$ and, if $d(u,v) \neq 2,$ then $v$ is not isolated in $\langle N(u) \rangle_G.$  The latter condition is that, if $u \in N(v),$ then $u$ and $v$ have a common neighbor; that is, $u$ is not isolated in $\langle N(v) \rangle_G.$
\end{proof}

\begin{cor}\label{C:nbduniversal}
The neighborhood function $N$ is a complete interference if and only if $G$ is point-determining, $|V|\ge2,$ $G$ has diameter at most $2,$ and every edge belongs to a triangle.
\end{cor}

\begin{proof}
By Proposition \ref{P:universality}, Corollary \ref{C:nbdsingleton} must apply to every $v \in V.$  The condition of that corollary is that if $u,v$ are neighbors, then there is a vertex adjacent to both.
\end{proof}

\begin{cor}\label{C:nbdallbut}
Assume $G$ is connected, let $v \in V,$ and let $D = V \setminus \{v\}.$  Then $N$ is an interference of $D$ if and only if $G$ is point-determining and has no isolated vertices and $v$ is adjacent to a vertex having degree at least two.
\end{cor}

\begin{proof}
For a point-determining graph $G$ without isolated vertices, $N$ is an interference of $D = V\setminus \{v\},$ if and only if $N(v)\cap N(D)\neq\emptyset.$ That is, if and only if $N(v)\cap N(u)\neq\emptyset$ for some $u\in V\setminus \{v\}.$ That is, if and only if there is a vertex $w$ adjacent to both $u$ and $v$ for some $u\in V\setminus \{v\}.$ That is, if and only if $v$ is adjacent to a vertex having degree at least two.
\end{proof}

The classes of graphs in the following examples illustrate Theorem \ref{T:nbd}.

\begin{ex}[Wheel]
The \emph{wheel} $W_n,$ $n\ge 3,$ is the graph obtained by taking a cycle $C_n$ and adjoining a vertex (the \emph{center}) adjacent to all cycle vertices.
For the wheel graph, $N$ is a complete interference.

\begin{proof}
Since $n\ge 3,$ every edge belongs to a triangle, and $W_n$ has diameter $2.$ By Corollary \ref{C:nbduniversal}, $N$ is a complete interference.
\end{proof}
\end{ex}

\begin{ex}[Windmill] \label{wm}
The \emph{windmill graph} $D_n^{(m)},$ $n\ge 3,$ is the graph obtained by taking $m$ copies of the complete graph $K_n,$ all the copies sharing exactly one vertex in common.
For the windmill, $N$ is a complete interference.

\begin{proof}
Since $n\ge 3,$ every edge of $D_n^{(m)}$ belongs to a triangle and $D_n^{(m)}$ has diameter $2.$ By Corollary \ref{C:nbduniversal}, $N$ is a complete interference.
\end{proof}
\end{ex}

\begin{ex}[Husimi Trees of Diameter Two] \label{BDA1} 
In particular, the windmill $D_3^m$ of Example \ref{wm} is well known as the `Friendship Graph'. In general, the argument in the proof of Example \ref{wm} can be extended to show that for the Husimi tree $F^{(m)}$ of diameter two, which consists of $m\ge 2$ complete blocks of arbitrary orders sharing exactly one common cut-vertex, $N$ is a complete interference.  
\end{ex}

\begin{ex}[Star Polygon]
A \emph{star} $n$-\emph{gon} is a graph obtained by replacing each edge of the cycle $C_n,$ $n\ge 3,$ by a triangle (see \cite{mcg}).
For the star $n$-gon, $N$ is an interference of $V(G)\setminus V(C_n)$ and of $V(C_n).$

\begin{proof}
The star $n$-gon is point-determining and has no isolated vertices. For every $u \in V,$ $d(u,D) \leq 1,$ where $D=V(C_n)$ or $V(G) \setminus V(C_n).$ Also, for every $u\in V\setminus D,$\  $N^2(u) \cap D \neq \eset.$ Hence by Theorem \ref{T:nbd}, $N$ is an interference of $V(G)\setminus V(C_n)$ and of $V(C_n).$
\end{proof}
\end{ex}

\begin{ex}[Helm and Crown]
The \emph{helm} $H_n$ is the graph obtained from the wheel $W_n$ by attaching a pendant edge at each vertex of the $n$-cycle.  The \emph{crown} $C_n\circ K_1$ is the graph obtained from a cycle $C_n$ by attaching a pendant edge to each vertex of the cycle.

Let the graph $H$ be any of the graphs $ H_n$ or $C_n\circ K_1.$  Then $N$ is an interference of $V(C_n) \subset V(H)$ as well as of the set of pendant vertices of $H.$

\begin{proof}
The graph $H$ is point-determining and has no isolated vertices.  Let $D=V(C_n)$  or the set of pendant vertices of $H.$  For every $u \in V \setminus D,$ $d(u,D) \leq 2.$

For $H_n,$ if $N^2(u) \cap D = \eset$ then $D=C_n$ and $u$ is the center of the wheel.  Then for every $v \in D,$ $uv$ is contained in a triangle.  For $C_n\circ K_1,$ $N^2(u) \cap D \neq \eset.$

Hence by Theorem \ref{T:nbd}, $N$ is an interference of $V(C_n)$ as well as of the set of pendant vertices of $H.$
\end{proof}
\end{ex}

A graph is said to be \emph{$2$-path-complete} if every two distinct vertices are joined by a path of length two. For instance, for any graph $H,$ the join $K_r+H,$ $r \geq2,$ is $2$-path complete.

\begin{ex}[$2$-Path-Complete Graphs]
The neighborhood function of every $2$-path-complete graph is a complete interference.

\begin{proof} Assume $G$ is $2$-path complete.  Every pair of distinct vertices are joined by a $2$-path and hence, in particular, for any edge $uv\in E(G)$ there exists $w\in N(u)\cap N(v).$  It follows that $N$ is a complete interference.
\end{proof}
\end{ex}

\subsection{Complemented neighborhood interference}

Now we discuss the interference properties of the complemented neighborhood function $\bar N.$

\begin{thm} \label{T:cnbd}
Let $G$ be a graph and let $\emptyset \neq D \subseteq V.$  Then $\bar N$ is an interference of $D$ if and only if $G$ is point-determining and every vertex $u \notin D$ that is adjacent to all vertices in $D$ has a nonneighbor that is not adjacent to all vertices in $D$ (equivalently, $u \in \big[ \bigcap_{v \in D} N(v) \big] \setminus D$ implies $\bar N(u) \not\subseteq \bigcap_{v \in D} N(v)$).
\end{thm}

\begin{proof}
The intersection $\bar N(u) \cap \bar N(v) = V \setminus [N(u) \cup N(v)],$ so it is nonempty if and only if $N(u) \cup N(v) \subset V.$

If $uv \notin E(G)$ (and $u \neq v$), then $N(u) \cup N(v) \subseteq V \setminus \{u,v\}$ so the nonemptiness condition is satisfied.  If $uv \in E(G),$ then $N(u) \cup N(v) = V$ if and only if $u, v$ are the centers of a spanning double star subgraph, or in other words, $d(x,\{u,v\}) \leq 1$ for every $x \in V,$ or in still other words, $N(\{u,v\}) = V,$ or yet again, $\{u,v\}$ is a dominating set for $G.$

Thus, in order that $N$ satisfy the overlap condition, either $D \not\subseteq N(u),$ or $D \subseteq N(u)$ and there is $v \in D$ that is not a dominating vertex of $G \setminus [N(u) \setminus v].$  In other words, for any vertex $u \notin D,$ either $D \not\subseteq N(u)$ or else $G$ does not contain the complete bipartite graph $K_{D,V \setminus N(u)}$ with bipartition $\{D,V\setminus N(u)\}.$

The statement $D \subseteq N(u)$ is equivalent to the statement that $u \in \bigcap_{v \in D} N(v).$  Thus, the overlap condition is equivalent to the property that, for every $u \in \big[ \bigcap_{v \in D} N(v) \big] \setminus D,$ $G$ does not contain $K_{D,V \setminus N(u)}.$  The latter property can be restated as that $V \setminus N(u) \not\subseteq \bigcap_{v \in D} N(v).$
\end{proof}

\begin{cor}\label{C:cnbdsingleton}
Given $G$ and $v \in V,$ $\bar N$ is an interference of $\{v\}$ if and only if $G$ is point-determining and, for every $u \in N(v),$ we have $N(u) \cup N(v) \neq V.$
\end{cor}

\begin{proof}
The property that $\bar N(u) \not\subseteq N(v)$ is equivalent to $N(u)  \cup N(v) \neq V.$
\end{proof}

\begin{cor}\label{C:compuni}
The neighborhood function $\bar N$ is a complete interference if and only if $G$ is point-determining and, for every edge $uv,$ $N(u) \cup N(v) \neq V.$
\end{cor}

\begin{proof}
By Proposition \ref{P:universality}, $\bar N$ is universal if and only if it is an interference of every singleton vertex set.  The result follows by Corollary \ref{C:cnbdsingleton}.
\end{proof}

The following example illustrates Theorem \ref{T:cnbd}.

\begin{ex}[Cycle Graph]
The neighborhood function $\bar N$ is a complete interference of $C_n$ if and only if $n \geq 5.$

\begin{proof}
The cycle $C_n$ is point-determining for $n=3$ and every $n\geq 5,$ but not for $n=4.$
If $n\geq 5,$  $N(u) \cup N(v) \neq V$ for every edge $uv$ because $|N(u) \cup N(v)| \leq 4 < n.$
If $n<5$ and $uv$ is an edge, then $N(u) \cup N(v)= V.$ Hence by Corollary \ref{C:compuni}, the result follows.
\end{proof}
\end{ex}

\begin{cor}\label{C:compreg}
Suppose $G$ is point-determining.  Then $\bar N$ is a complete interference if
\begin{enumerate}
\item $G$ is regular of degree $k$ and $n > 2k,$ or
\label{C:compreg1}
\item the sum of any two degrees in $G$ is $< n,$ or
\label{C:compreg2}
\item the sum of degrees of any two vertices at distance $2$ is $\leq n$ and the sum of degrees of any two other vertices is $< n.$
\label{C:compreg3}
\end{enumerate}
\end{cor}

The conditions are listed in order of decreasing simplicity and increasing generality.

\begin{proof}
Clearly, \eqref{C:compreg1} $\Rightarrow$ \eqref{C:compreg2} $\Rightarrow$ \eqref{C:compreg3}.

Suppose \eqref{C:compreg3} holds.  Then $|N(u) \cup N(v)| \leq d(u) + d(v) < n$ if $d(u,v)\ne2.$  If $d(u,v)=2,$ then $|N(u) \cup N(v)| \leq d(u) + d(v) - 1 \leq n,$ since $u$ and $v$ have a common neighbor.  In either case, the union of neighborhoods is smaller than $V.$
\end{proof}

As examples of Corollary \ref{C:compreg} \eqref{C:compreg1},  $\bar N$ is a complete interference if $G = C_n$ where $n\geq5$ or if $G$ is cubic of order $n > 6.$

\subsection{Line graphs}

In this section, we deal with graphs $G$ having no isolated vertices. We use the notation $L(G)$ for the line graph of $G,$ $M_k$ for a matching of $k$ edges, and $P_n$ for a path of order $n.$

The notations $N_L$ and $\bar N_L$ denote the neighborhood and complemented neighborhood functions for $L(G),$ i.e., $N_L(e) := N_{L(G)}(e)$ and $\bar N_L(e) := E \setminus N_{L(G)}(e)$ for $e \in E.$

\begin{lem}\label{L:lginjectivity}
For a graph $G,$  $N_L,$ and also $\bar N_L,$ is injective if and only if $G$ has at most one isolated edge and no component $G'$ of $G$ satisfies $P_4 \subseteq G' \subseteq K_4.$
\end{lem}

\begin{proof}  Assume that $G$ is connected and $N_L$ is not injective. Then there exist edges $e_1$ and $e_2$ such that $N_L(e_1)=N_L(e_2).$ Then $e_1$ and $e_2$ are not adjacent. For, if $e_1$ and $e_2$ are  adjacent, $e_1\in N_L(e_2)$ and $e_1\not \in N_L(e_1),$ a contradiction.  $|N_L(e_1)| = |N_L(e_2)|$ is the number of edges in $G$ adjacent to both $e_1$ and $e_2.$ Hence, $0 \leq |N_L(e_1)|= |N_L(e_2)|\leq 4.$  If $|N_L(e_1)|= 0,$ $G\cong M_2,$ which is not connected. If $|N_L(e_1)|= 1,$ $G\cong P_4.$ If $|N_L(e_1)|= 2,$ $G\cong C_4$ or $K_3$ with a pendant edge.  If $|N_L(e_1)|=3,$ $G\cong K_4-e.$ If $|N_L(e_1)|= 4,$ $G\cong K_4.$ Hence $P_4 \subseteq G \subseteq K_4.$

For the converse, let $P_4 \subseteq G \subseteq K_4.$ Then for any two nonadjacent vertices $e_1$ and $e_2$ of $L(G),$  $N_L(e_1)=N_L(e_2).$ Therefore, $N_L$ is not injective.

If $G$ is disconnected, then only the existence of two isolated edges or a component $G'$ such that $P_4 \subseteq G' \subseteq K_4$ can produce two edges with the same neighborhood.

The function $\bar N_L$ is injective if and only if $N_L$ is injective. Therefore,  $\bar N_L$ is not injective if and only if $P_4 \subseteq G \subseteq K_4.$
\end{proof}

\begin{lem}\label{L:lgdiscinjectivity}
For a disconnected graph $G,$ $N_L,$ and also $\bar N_L,$ is injective in $L(G)$ if and only if it is injective in the line graph of each component of $G$ and no two components of $G$ are $K_2.$
\end{lem}

\begin{proof}
Suppose $N_L$ is injective on each component.  Let $G'$ be the component that contains edge $e.$  Then $N_L(e)$ is contained in $E(G'),$ which is disjoint from the edge set of any other component.  The only way for $N_L(e)$ to equal $N_L(f)$ is for $e,f$ to be in different components and have no neighboring edges, but then the two components are single edges.
\end{proof}

Theorem \ref{T:nbd}(b) simplifies for line graphs.  Let $d_L(\cdot,\cdot)$ be the distance function in $L(G),$ that is, the distance between edges or edge sets in $G.$  Let $d_L(e)$ be the degree in $L(G),$ i.e., the degree of an edge in $G,$ and let $d(u)$ be the degree of a vertex in $G.$

\begin{thm}\label{C:nbdlg}
Let $G$ be a graph with nonempty edge set and let $\eset \neq D \subseteq E(G).$  Then $N_L$ is an interference of $D$ in $L(G)$ if and only if no component $G'$ of $G$ is $K_2$ or satisfies $P_4 \subseteq G' \subseteq K_4$ and, for every $e = uv \in E(G) \setminus D,$ at least one neighboring edge of $e$ has a neighboring edge in $D.$
\end{thm}

For instance, Theorem \ref{C:nbdlg} applies if $G$ is connected with order $\geq 5.$

\begin{proof}
First, suppose $N_L$ is an interference of $D.$  Then we may assume $G$ is connected (by Lemma \ref{L:lgdiscinjectivity}), does not satisfy $P_4 \subseteq G \subseteq K_4$ (by Lemma \ref{L:lginjectivity}), and has order $n\geq3,$ since if $G=K_2,$ $N_L(e)=\eset.$

If $G = K_3$ or $K_{1,3},$ $L(G)=K_3$ and $N_L$ is a complete interference by Corollary \ref{C:nbduniversal}.  Thus, we may assume $G$ has order $n \geq 5.$

By Theorem \ref{T:nbd}, $N_L$ is an interference of $D$ if and only if, for every $e = uv \in E(G)\setminus D,$ (a) $d_L(e,D)\leq 2$ and  (b) if $N_L^2(e) \cap D = \emptyset,$ there is an edge $f\in D$ such that $e$ and $f$ are contained in a triangle or a star in $G.$ The conditions (a) and (b) together are equivalent to stating that some neighboring edge of $e$ has a neighboring edge in $D.$

Now, suppose no component $G'$ of $G$ is $K_2$ or satisfies $P_4 \subseteq G' \subseteq K_4$ and, for every $e = uv \in E(G) \setminus D,$ at least one neighboring edge of $e$ has a neighboring edge in $D.$  By Lemmas \ref{L:lginjectivity} and \ref{L:lgdiscinjectivity}, $N_L$ is injective.  By assumption, $d(e,D)\le2.$  If no edge of $D$ has distance $2$ from $e,$ then every such edge is a neighbor of $e.$  Let $f$ be a neighbor of $e$ that has a neighbor $d\in D.$  Then $d$ is a neighbor of $e$ as well, so $d,e,f$ form a triangle in $L(G).$  Hence, (a) and (b) of Theorem \ref{T:nbd} are satisfied and $N_L$ is an interference of $D$ in the line graph.
\end{proof}

\begin{cor}\label{C:nbdsingletonlg}
For any graph $G$ (without isolated vertices), $N_L$ is an interference of $\{f\}$ in $L(G)$ if and only if $G$ is connected and has at least two edges, $G$ does not satisfy $P_4 \subseteq G \subseteq K_4,$ and every edge is a neighbor of a neighbor of $f.$
\end{cor}

\begin{proof}
This follows by applying Theorem \ref{C:nbdlg} to $D=\{f\}.$
The assumption about neighboring edges implies that $G$ is connected and, by applying it to $f$ itself, $G$ has at least one edge other than $f.$
\end{proof}

\begin{cor}\label{C:nbdlguniversal}
Let $G$ be a connected graph of order at least $3.$  Then $N_L$ is a complete interference for $E(G)$ if and only if $G$ does not satisfy $P_4 \subseteq G \subseteq K_4,$ $L(G)$ has diameter at most $2,$ every pendant edge in $G$ has an endpoint of degree at least $3,$ and every edge in $G$ with endpoints of degree $w$ belongs to a triangle.
\end{cor}

A more complete result would follow upon characterizing graphs whose line graphs have diameter at most $2.$  An equivalent condition is that $G$ has no induced $2$-edge matching.  A full characterization is an unsolved problem.

\begin{proof}
This follows by simplifying Corollary \ref{C:nbduniversal} applied to $L(G).$  The assumptions about pendant edges and triangles imply that $G$ has no isolated edges.
\end{proof}

\begin{cor}\label{C:cnbdlg}
Let $G$ be a connected graph of order $n \geq 5,$ and let $\eset \neq D \subseteq E(G).$  Then $\bar N_L$ is an interference of $D$ if, and only if, for every edge $e \notin D$ that is adjacent to all edges in $D,$ there exists an edge $f,$ not adjacent to $e,$ such that $D$ is not composed solely of edges with one end vertex in $V(e)$ and the other in $V(f).$
\end{cor}

\begin{proof}
The first hypotheses ensure that $\bar N_L$ is an interference of $D$ if and only if, for every edge $e \notin D$ that is adjacent to all edges in $D,$ there is an edge $f,$ not adjacent to $e,$ such that $f$ is not adjacent to all edges in $D.$  That is the statement of the corollary.
\end{proof}

\begin{cor}\label{C:cnbdlgsize}
Let $G$ be a connected graph of order $n \geq 5,$ and let $\eset \neq D \subseteq E(G)$ consist of at least five edges.  Then either $\bar N_L$ is an interference of $D,$ or there is an edge $e \notin D$ that is adjacent to every other edge in $G.$
\end{cor}

We want a criterion for when a general line graph $L(G)$ has vertices $e,f$ whose neighborhood union is the entire vertex set, $V_L = E(G).$  This means that in $G$ every edge has a vertex in common with $e$ or $f.$  Deleting $V(\{e,f\})$ leaves an independent set in $G.$  Therefore:

\begin{prop}\label{P:cnbdlgindep}
If $G$ is a connected graph of order $n$ whose independence number $\alpha(G) < n-4,$ then $\bar N_L$ is a complete interference for $E(G).$
\end{prop}

\begin{proof}
Since $\alpha(G) > 0,$ we have $n \geq 6;$ therefore by Lemma \ref{L:lginjectivity} $\bar N_L$ is injective.  As explained before the statement, $\bar N_L(e) \cup \bar N_L(f) \neq E(G)$ for any two edges $e,f.$  Thus, $\bar N_L$ is a complete interference.
\end{proof}

Suppose $G$ is $k$-regular where $k > 1.$  The degree $d_L(e) = 2(k-1),$ so by Corollary \ref{C:compreg} and Lemma \ref{L:lginjectivity}, $\bar N_L$ is a complete interference if $n_L = |V(L(G)| = |E(G)| > 2 \cdot 2(k-1).$  We know $n_L = \frac12 nk,$ so we have the following result:

\begin{cor}\label{C:compreglg}
Let $G$ be a regular, connected graph of order $n \geq 8.$  Then $E(G)$ has the complete interference $\bar N_L.$
\end{cor}

\begin{proof}
Connectedness implies $k\geq2$ and therefore $L(G)$ has degree $2(k-1) \geq 2.$  The inequality $n_L = \frac12 nk > 2 \cdot 2(k-1)$ has the solution $n > 8(k-1)/k,$ or $n \geq 8.$  Since $n > 4$ and $G$ is connected, by Lemma \ref{L:lginjectivity} $L(G)$ is point-determining.  Thus, by Corollary \ref{C:compreg}(1), $\bar N_L$ is a complete interference.
\end{proof}

\section{Distance-pattern distinguishing sets}

We continue to assume that the interference graph $I$ is complete and that $G$ is another graph on vertex set $V.$
Given an arbitrary nonempty subset $M$ of vertices in $G,$  each vertex $u$ is associated with the set $f_{M}(u)=\{d(u,v):v\in M\}.$
The function $f_M$ is called the $M$-\emph{distance pattern} of $G.$

\begin{dfn} [{\cite{kag,kag1}}] \label{df.3}  If, for a subset $M$ of vertices in a graph $G,$ $f_M$ is injective, then the set $M$ is called a \emph{distance-pattern distinguishing set} (DPD-set in short) of $G.$ \hfill $\blacksquare$
\end{dfn}

\begin{lem}
For a graph $G$ of order at least $2,$ the $M$-distance pattern $f_M$ is not an interference of any set of cardinality one.
\end{lem}

\begin{proof}
Let $M=\{v\}.$  Assume that $f_M$ is an interference of the set $M=\{v\}.$
Then $f_M(u)=\{d(u,v)\}$ and therefore if $u \notin M$ exists, $f_M(u)\cap f_M(v)= \emptyset,$ contradicting the definition of interference.
\end{proof}

For $D \subseteq V,$ define $d(u,D) = \min_{v \in D} d(u,v)$ for each vertex $u \in V;$ this is the distance from $u$ to $D.$

\begin{thm}
For a path $P_n=(v_1, v_2, \dots, v_n),\ n\ge 4,$ the $M$-distance pattern is an interference of the set $M=\{v_1, v_2, v_4, v_7, \dots, v_{1+r(r-1)/2}\},$ where $r=|M|=\big\lceil \frac{1+\sqrt{8n-7}}{2}\big\rceil.$
\end{thm}

\begin{proof}
The set $M$ is a DPD-set of $P_n$ (see \cite{kag1}).  The distance between the $(j-1)^{th}$ and $j^{th}$ elements of $M$ is $j-1.$ Therefore, $f_M(M)= \{1, 2, 3, \dots, r-1\}.$  For any $w \notin M,$ $d(w,M)\leq r-1,$ hence, $f_M(w)\cap f_M(M)\neq \emptyset.$ Therefore,  $f_M$ is an interference of $M.$
Also, ${1+\frac{r(r-1)}{2}}\le n,$ which implies  $r=\big\lceil \frac{1+\sqrt{8n-7}}{2}\big\rceil.$
\end{proof}

\section{Conclusion}

We conclude with open questions about interference.

\subsection{Interference index}

As we mentioned in Section \ref{index}, although the interference index, or universal interference index, is always of the form $\lceil \log_2(n+r)\rceil$ for $1 \le r \le n,$ we do not know the appropriate value of $r$ for many families of graphs, including several examples in Section \ref{nbd}.  We have partial results for complete bipartite graphs, which depend on solving a new problem of extremal set theory.  It seems likely that the behavior of complete multipartite graphs is similar; that is, partial results can be obtained but exact answers depend on finding new results in extremal set theory.  It would be interesting to see what these problems are and try to answer them.  They appear to be difficult.

\subsection{Neighborhood-based interference}

We somewhat arbitrarily chose to study neighborhood-based interference with respect only to the complete interference graph, $I=K_n.$  How do the interference properties of $N$ and $\bar N$ change if we choose $G=I$?  Or, if we choose other interference graphs?
Or, if we take specific graphs like the $n$-cube $Q_n,$ the rectangular lattice grid $P_m\times P_n,$ or any one of the polyhedral graphs for interference graphs? Or, if we take only a planar (or, outerplanar) graph for an interference graph?   

\subsection{Other interference functions}

We found that the neighborhood function and its complement have interesting interference properties.  What other natural graph functions are similarly interesting?  For instance, Steven Hedetniemi mentioned (in private correspondence) that the closed neighborhood function $f(u) := N_G[u] := N_G(u) \cup \{u\}$ is a universal interference with respect to the interference graph $I=G.$  
What happens if the interference graph is complete?


\begin{thebibliography}{9}

\bibitem{bda1} B.D.\ Acharya,  \textbf{Set-valuations of Graphs and Their Applications}, MRI Lecture notes in Applied Mathematics, No.2, The Mehta Research Institute of Mathematics and Mathematical Physics, Allahabad, 1983.

\bibitem{And} I.\ Anderson,  \textbf{Combinatorics of Finite Sets}, Clarendon Press, Oxford, 1987.  Corr.\ repr.: Dover, Mineola, N.Y., 2002.

\bibitem{kag} K.A.\ Germina, \emph{Distance-patterns of vertices in a graph}, International Mathematical Forum, \textbf{5}(34) (2010), 1697--1704.

\bibitem{kag1} K.A.\ Germina, \emph{Set-valuations of Graphs and Their Applications}, Final Technical Report, Grant-in-Aid project No.SR/S4/MS:277/05, funded by the Department of Science and Technology (DST), Govt.\ of India, April 2009.

\bibitem{mcg} M.C.\ Golumbic, \textbf{Algorithmic Graph Theory and Perfect Graphs}, Academic Press, San Diego, 1980.

\bibitem{fh}   F.\ Harary, \textbf{Graph Theory}, Addison--Wesley, Reading, Massachusetts, 1969.

\bibitem{hhs} T.\ Haynes, S.T.\ Hedetneimi and P.J.\ Slater, \textbf{Fundamentals of Domination in Graphs}, Marcel Dekker, 1998.

\bibitem{dps} D.P.\ Sumner, \emph{Point determination in graphs}, Discrete Mathematics, \textbf{5}(1973), 179--187.

\end{thebibliography}
\end{document}